\documentclass[10pt,leqno,a4paper]{article}
\usepackage{amsfonts}
\usepackage{amsthm,amsmath}
\usepackage{enumerate}
\usepackage[unicode,colorlinks,linkcolor=blue,citecolor=red]{hyperref}
\newtheorem{definition}{Definition}
\newtheorem{theorem}{Theorem}
\newtheorem{proposition}{Proposition}
\newtheorem{cor}{Corollary}
\theoremstyle{definition}
\newtheorem*{rem}{Remark}
\newtheorem{ex}{Example}

\DeclareMathOperator\cli{cl}
\newcommand\ga{\alpha}
\newcommand\gb{\beta}
\newcommand\gd{\delta}
\newcommand\gf{\varphi}
\newcommand\MSC{\textbf{2000 Mathematics Subject Classification: }}
\newcommand\keyw{\textbf{Key words: }}
\begin{document}

\title{IMPROVEMENT OF GRAPH THEORY WEI'S INEQUALITY%
\thanks{\MSC 05C35\protect}
\thanks{\keyw clique number, degree sequence}}
\author{Nedyalko Dimov Nenov}
\maketitle

\begin{abstract}
Wei in \cite8 and \cite9 discovered a bound on the clique number of
a given graph in terms of its degree sequence. In this note we give
an improvement of this result.
\end{abstract}

We consider only finite non-oriented graphs without loops and multiple edges.
A set of $p$ vertices of a graph is called a $p$-clique if each two of them
are adjacent. The greatest positive integer $p$ for which $G$ has a $p$-clique
is called clique number of $G$ and is denoted by $\cli(G)$. A set of vertices
of a graph is independent if the vertices are pairwise nonadjacent.
The independence number $\ga(G)$ of a graph $G$ is the cardinality of
a largest independent set of $G$.

In this note we shall use the following notations:
\begin{itemize}
\item
$V(G)$ is the vertex set of graph $G$;
\item
$N(v)$, $v\in V(G)$ is the set of all vertices of $G$ adjacent to $v$;
\item
$N(V)$, $V\subseteq V(G)$ is the set $\bigcap_{v\in V} N(v)$;
\item
$d(v)$, $v\in V(G)$ is the degree of the vertex $v$, i.e. $d(v)=|N(v)|$.
\end{itemize}

Let $G$ be a graph, $|V(G)|=n$ and $V\subseteq V(G)$. We define
\begin{align*}
W(V)&=\sum_{v\in V}\frac1{n-d(v)};\\
W(G)&=W(V(G)).
\end{align*}

Wei in \cite8 and \cite9 discovered the inequality
\[
\ga(G)\ge\sum_{v\in V(G)}\frac1{1+d(v)}.
\]

Applying this inequality to the complementary graph of $G$ we see
that it is equivalent to the following inequality
\[
\cli(G)\ge\sum_{v\in V(G)}\frac1{n-d(v)}
\]
that is
\begin{equation}\label{NN:1}
\cli(G)\ge W(G).
\end{equation}
Alon and Spencer \cite1 gave an elegant probabilistic proof of Wei's inequality.
In the present note we shall improve the inequality~\eqref{NN:1}.

\begin{definition}\label{d:1}
Let $G$ be a graph, $|V(G)|=n$ and $V\subseteq V(G)$.
The set $V$ is called a $\gd$-set in $G$, if
\[
d(v)\le n-|V|
\]
for all $v\in V$.
\end{definition}

\begin{ex}\label{ex:1}
Any independent set $V$ of vertices of a graph $G$ is a $\gd$-set in $G$ since
$N(v)\subseteq V(G)\setminus V$ for all $v\in V$.
\end{ex}

\begin{ex}\label{ex:2}
Let $V\subseteq V(G)$ and $|V|\ge\max\{d(v),v\in V(G)\}$. Since $d(v)\le|V|$
for all $v\in V(G)$, $V(G)\setminus V$, is a $\gd$-set in $G$.
\end{ex}

The next statement obviously follows from Definition~\ref{d:1}:

\begin{proposition}\label{pro:1}
Let $V$ be a $\gd$-set in a graph $G$. Then $W(V)\le 1$.
\end{proposition}

\begin{definition}\label{d:2}
A graph $G$ is called an $r$-partite graph if
\[
V(G)=V_1\cup\dots\cup V_r,
\quad
V_i\cap V_j=\emptyset,
\quad
i\ne j,
\]
where the sets $V_i$, $i=1,\dots,r$, are independent.
If the sets $V_i$, $i=1,\dots,r$, are $\gd$-sets in $G$, then $G$ is called
generalized $r$-partite graph. The smallest integer $r$ such that $G$ is
a generalized $r$-partite graph is denoted by $\gf(G)$.
\end{definition}

\begin{proposition}\label{pro:2}
$\gf(G)\ge W(G)$.
\end{proposition}

\begin{proof}
Let $\gf(G)=r$ and
\[
V(G)=V_1\cup\dots\cup V_r,
\quad
V_i\cap V_j=\emptyset,
\quad
i\ne j,
\]
where $V_i$, $i=1,\dots,r$, are $\gd$-sets in $G$.
Since $V_i\cap V_j=\emptyset$, $i\ne j$, we have
\[
W(G)=\sum_{i=1}^r W(V_i).
\]
According to Proposition~1 $W(V_i)\le 1$, $i=1,\dots,r$.
Thus $W(G)\le r=\gf(G)$.
\end{proof}

Below (see Theorem~\ref{th:1}) we shall prove that $\cli(G)\ge\gf(G)$.
Thus~\eqref{NN:1} follows from Proposition~2.

\begin{definition}[\cite2]\label{d:3}
Let $G$ be a graph and $v_1,\dots,v_r\in V(G)$.
The sequence $v_1,\dots,v_r$ is called an $\ga$-sequence in $G$ if
the following conditions are satisfied:
\begin{enumerate}[\rm(i)]
\item
$d(v_1)=\max\{d(v)\mid v\in V(G)\}$;
\item
$v_i\in N(v_1,\dots,v_{i-1})$ and $v_i$ has maximal degree in the graph\linebreak
$G[N(v_1,\dots,v_{i-1})]$, $2\le i\le r$.
\end{enumerate}
\end{definition}

Every $\ga$-sequence $v_1,\dots,v_s$ in the graph $G$ can be extended to
an $\ga$-sequence $v_1,\dots$, $v_s,\dots,v_r$ such that $N(v_1,\dots,v_{r-1})$
be a $\gd$-set in $G$. Indeed, if the $\ga$-sequence $v_1,\dots,v_s$, $\dots,v_r$
is such that it is not continued in a $(r+1)$-clique
(i.e. $v_1,\dots,v_s,\dots,v_r$ is a maximal $\ga$-sequence in the sense of
inclusion) then $N(v_1,\dots,v_{r-1})$ is an independent set and, therefore,
a $\gd$-set in $G$. However, there are $\ga$-sequences $v_1,\dots,v_r$ such
that $N(v_1,\dots,v_{r-1})$ is a $\gd$-set but it is not an independent set.

\begin{theorem}\label{th:1}
Let $G$ be a graph and $v_1,\dots,v_r$, $r\ge 2$,
be an $\ga$-sequence in $G$ such that $N(v_1,\dots,v_{r-1})$ is a $\gd$-set
in $G$. Then

$(a)$ $\gf(G)\le r\le\cli(G)$;

$(b)$ $r\ge W(G)$.
\end{theorem}

\begin{proof}
According to Definition~\ref{d:3} $v_1,\dots,v_r$ is an $r$-clique and thus
$r\le\cli(G)$. Since $N(v_1,\dots,v_{r-1})$ is a $\gd$-set, the graph $G$
is a generalized $r$-partite graph, \cite6. Hence $r\ge\gf(G)$.
The inequality~(b) follows from (a) and Proposition~2.
\end{proof}

\begin{rem}
Theorem~\ref{th:1} (b) was proved in \cite7 in the special case when\linebreak
$N(v_1,\dots,v_{r-1})$ is independent set in $G$.
\end{rem}

\begin{definition}\label{d:4}
Let $G$ be a graph and $v_1,\dots,v_r\in V(G)$.
The sequence $v_1,\dots,v_r$ is called $\gb$-sequence in $G$ if the following
conditions are satisfied:
\begin{enumerate}[\rm (i)]
\item
$d(v_1)=\max\{d(v)\mid v\in V(G)\}$;
\item
$v_i\in N(v_1,\dots,v_{i-1})$ and
$d(v_i)=\max\{d(v)\mid v\in N(v_1,\dots,v_{r-1})\}$, $2\le i\le r$.
\end{enumerate}
\end{definition}

\begin{theorem}\label{th:2}
Let $v_1,\dots,v_r$ be a $\gb$-sequence in a graph $G$
such that
\[
d(v_1)+\dots+d(v_r)\le (r-1)n,
\]
where $n=|V(G)|$. Then $r\ge W(G)$.
\end{theorem}

\begin{proof}
According to \cite5 it follows from $d(v_1)+\dots+d(v_r)\le (r-1)n$ that $G$
is a generalized $r$-partite graph. Hence $r\ge\gf(G)$ and Theorem~\ref{th:2}
follows from Proposition~\ref{pro:2}.
\end{proof}

\begin{cor}
Let $G$ be a graph, $|V(G)|=n$ and $v_1,\dots,v_r$ be
a $\gb$-sequence in $G$ which is not contained in $(r+1)$-clique.
Then $r\ge W(G)$.
\end{cor}

\begin{proof}
Since $v_1,\dots,v_r$ is not contained in $(r+1)$-clique it follows that
$d(v_1)+\dots+d(v_r)\le (r-1)n$, \cite3.
\end{proof}

\begin{theorem}\label{th:3}
Let $G$ be a graph, $|V(G)|=n$ and $v_1,\dots,v_r$, $r\ge 2$,
be a $\gb$-sequence in $G$ such that $N(v_1,\dots,v_{r-1})$ is a $\gd$-set
in $G$. Then $r\ge W(G)$.
\end{theorem}

\begin{proof}
Since $N(v_1,\dots,v_{r-1})$ is a $\gd$-set according to \cite6 there exists
an $r$-partition
\[
V(G)=V_1\cup\dots\cup V_r,
\qquad
V_i\cap V_j=\emptyset,
\quad
i\ne j,
\]
where $V_i$, $i=1,\dots,r$, are $\gd$-sets and $v_i\in V_i$.
Thus, we have
\[
d(v_i)\le n-|V_i|,
\quad
i=1,\dots,r.
\]

Summing up these inequalities we obtain that $d(v_1)+\dots+d(v_r)\le (r-1)n$.
Therefore Theorem~\ref{th:3} follows from Theorem~\ref{th:2}.
\end{proof}

\noindent
\begin{tabular}[t]{@{}l@{}}
Nedyalko Dimov Nenov\\
Faculty of Mathematics and Informatics\\
St Kliment Ofridski University of Sofia\\
5, James Bourchier Blvd.\\
BG-1164 Sofia, Bulgaria\\
e-mail: \texttt{nenov@fmi.uni-sofia.bg}
\end{tabular}
\end{document}